\documentclass[a4paper,12pt]{article}
\usepackage[T1]{fontenc}
\usepackage[english]{babel}
\usepackage[dvips]{graphicx}
\usepackage{amssymb,amsthm,amsmath}

\renewcommand{\Im}{\textnormal{Im}\,}

\newcommand{\p}{\varphi}
\newcommand{\PP}{ { \R^2 \setminus\{ 0 \} } }
\newcommand{\PS}{ { \Rn \setminus\{ 0 \} } }
\newcommand{\N}{\mathbb{N}}
\newcommand{\R}{\mathbb{R}}
\newcommand{\C}{\mathbb{C}}
\newcommand{\Rn}{ {\mathbb{R}^n} }
\newcommand{\Ball}[3]{B_{#1}(#2,#3)}

\newcommand{\comment}[1]{}

\newtheorem{theorem}[equation]{Theorem}
\newtheorem{lemma}[equation]{Lemma}
\newtheorem{proposition}[equation]{Proposition}
\newtheorem{corollary}[equation]{Corollary}

\theoremstyle{definition}
\newtheorem{remark}[equation]{Remark}

\numberwithin{equation}{section}

\newcounter{minutes}
\divide\time by 60
\newcounter{hours}
\multiply\time by 60 \addtocounter{minutes}{-\time}

\begin{document}

\begin{center}
{\Large \bf Close-to-convexity of quasihyperbolic and $j$-metric balls}
\end{center}
\medskip

\begin{center}
{\large Riku Klén }
\end{center}
\bigskip

\begin{abstract}
  We will consider close-to-convexity of the metric balls defined by the quasihyperbolic metric and the $j$-metric. We will show that the $j$-metric balls with small radii are close-to-convex in general subdomains of $\Rn$ and the quasihyperbolic balls with small radii are close-to-convex in the punctured space.

  \hspace{5mm}


\end{abstract}
\bigskip
\begin{center}
\texttt{File:~\jobname .tex, \number\year-\number\month-\number\day, \thehours.\ifnum\theminutes<10{0}\fi\theminutes}
\end{center}

\section{Introduction}

Let $G$ be a proper subdomain of the Euclidean space $\Rn$, $n \ge 2$. We define the \emph{quasihyperbolic length} of a rectifiable arc $\gamma \subset G$ by
\[
  \ell_k(\gamma) = \int_{\alpha}\frac{|dz|}{d(z,\partial G)}.
\]
The \emph{quasihyperbolic metric} is defined by
\[
  k_G(x,y) = \inf_\gamma \ell_k(\gamma),\index{hyperbolic metric, $k_G$}
\]
where the infimum is taken over all rectifiable curves in $G$ joining $x$ and $y$. If the domain $G$ is understood from the context we use notation $k$ instead of $k_G$.

The \emph{distance ratio metric} or \emph{$j$-metric} in a proper subdomain $G$ of the Euclidean space $\Rn$, $n \ge 2$, is defined by
\[
  j_G(x,y) = \log \left( 1+\frac{|x-y|}{\min \{ d(x),d(y) \}}\right),
\]
where $d(x)$ is the Euclidean distance between $x$ and $\partial G$. If the domain $G$ is understood from the context we use notation $j$ instead of $j_G$.

The quasihyperbolic metric was discovered by F.W. Gehring and B.P. Palka in 1976 \cite{gp} and the distance ratio metric by F.W. Gehring and B.G. Osgood in 1979 \cite{go}. The above form for the distance ratio metric was introduced by M. Vuorinen \cite{vu1}. For the quasihyperbolic metric the explicit formula is known in a few special domains like half-space, where it agrees with the usual hyperbolic metric, and punctured space, where the formula was introduced by G.J. Martin and B.G. Osgood in 1986 \cite{mo}. They proved that for $x,y \in \PS$ and $n \ge 2$
\begin{equation}\label{moeq}
  k_\PS(x,y) = \sqrt{\alpha^2+\log^2\frac{|x|}{|y|}},
\end{equation}
where $\alpha \in [0,\pi]$ is the angle between line segments $[x,0]$ and $[0,y]$ at the origin.

This paper is a continuation of \cite{k1,k2} and we will study the open problem posed by M. Vuorinen in \cite{vu2} and reformulated in \cite{k3}. The idea is to obtain more knowledge about the geometry of the metric balls defined by the quasihyperbolic and the distance ratio metrics. Before introducing the main results we define the metric ball and some geometric properties of domains.

For a domain $G \subsetneq \Rn$ and a metric $m \in \{ k_G,j_G \}$ we define the \emph{metric ball} (\emph{metric disk} in the case $n = 2$) for $r > 0$ and $x \in G$ by
\[
  \Ball{m}{x}{r} = \{ y \in G \colon m(x,y) < r \}.
\]
We call $\Ball{k}{x}{r}$ the \emph{quasihyperbolic ball} and $\Ball{j}{x}{r}$ the \emph{$j$-metric ball}. A domain $G \subsetneq \Rn$ is \emph{starlike with respect to} $x \in G$ if for all $y \in G$ the line segment $[x,y]$ is contained in $G$ and $G$ is \emph{strictly starlike with respect to} $x$ if each half-line from the point $x$ meets $\partial G$ at exactly one point. If $G$ is (strictly) starlike with respect to all $x \in G$ then it is \emph{(strictly) convex}. A domain $G \subsetneq \Rn$ is \emph{close-to-convex} if $\Rn \setminus G$ can be covered with non-intersecting half-lines. By half-lines we mean sets $\{ x \in \Rn \colon x = t y+z, \, t > 0 \}$ and $\{ x \in \Rn \colon x = t y+z, \, t \ge 0 \}$ for $z \in \Rn $ and $y \in \PS$. Clearly convex domains are starlike and starlike domains are close-to-convex as well as complements of close-to-convex domains are starlike with respect to infinity. However, close-to-convex sets need not be connected. An example of a close-to-convex disconnected set is the union of two disjoint convex domains. We use notation $B^n(x,r)$ and $S^{n-1}(x,r)$ for Euclidean balls and spheres, respectively, with radius $r > 0$ and center $x \in \Rn$. We abbreviate $B^n(r) = B^n(0,r)$, $B^n = B^n(1)$, $S^{n-1}(r)=S^{n-1}(0,r)$ and $S^{n-1}=S^{n-1}(1)$.

The study of local convexity properties of metric balls for hyperbolic type metrics was initiated by M. Vuorinen. He posed \cite{vu2} the following problem for hyperbolic type metrics $m$ in a domain $G \subsetneq \Rn$:

\vspace{5mm}
\noindent Does there exist a radius $r_0 > 0$ such that $\Ball{m}{x}{r}$ is convex (in Euclidean geometry) for all $x \in G$ and $r \in (0,r_0]$?
\vspace{5mm}

\noindent If such a radius exists we call it the \emph{radius of convexity for the metric ball}. O. Martio and J. Väisälä showed \cite{mv} that quasihyperbolic balls are always convex in convex domains. Soon after it was shown \cite{k2} that quasihyperbolic balls are convex in punctured space $\Rn \setminus \{ 0 \}$ whenever the radius is less than or equal to 1. Recently J. Väisälä proved \cite{v2} that in the case $n=2$ the quasihyperbolic disks of every plane domain are convex whenever the radius is less than or equal to 1. The $j$-metric balls in a general domain are convex whenever $r \le \log 2$ and they are always convex in convex domains \cite{k1}.

It is natural to ask whether we can replace the convexity by starlikeness or some other geometric property. It turns out that the quasihyperbolic balls and the $j$-metric are always starlike with respect to $x$ in domains that are starlike with respect to $x$ \cite{k1,k2}. In a general domain the $j$-metric balls $\Ball{j}{x}{r}$ are starlike with respect to $x$ whenever $r \le \log (1+\sqrt{2})$, \cite{k1} and the quasihyperbolic balls $\Ball{k}{x}{r}$ are starlike with respect to $x$ if $r \le \pi/2$, \cite{v1}. The latter bound is presumably not sharp and in the punctured space the quasihyperbolic balls $\Ball{k}{x}{r}$ are starlike with respect to $x$ whenever $r \le \kappa$, where $\kappa \approx 2.83$, \cite{k2}.

In this paper we will concentrate on the close-to-convexity of the quasihyperbolic and $j$-metric balls. For the quasihyperbolic and the $j$-metric balls the radius of convexity and starlikeness are independent of the dimension $n$. However, the radius of close-to-convexity is different for the cases $n=2$ and $n>2$. Our main result is the following theorem.
\begin{theorem}\label{mainthm}
  For a domain $G \subsetneq \Rn$ and $x \in G$ the $j$-metric ball $\Ball{j}{x}{r}$ is close-to-convex, if $r \in (0, \log(1+\sqrt{3})]$.

  For $y \in \PS$ the quasihyperbolic ball $\Ball{k}{y}{r}$ is close-to-convex, if $r \in (0,\lambda]$, where $\lambda$ is defined by (\ref{lambda}) and has a numerical approximation $\lambda \approx 2.97169$.

  Moreover, the constants $\log(1+\sqrt{3})$ and $\lambda$ are sharp in the case $n = 2$.
\end{theorem}

\section{Close-to-convexity of $j$-metric balls}

In this section we will consider the close-to-convexity of $j$-metric balls. We will begin with a simple domain $\PS$ and extend the result to a general subdomain of $\Rn$.
\begin{theorem}\label{jforPS}
  Let $G = \PS$ and $x \in G$. For $r \in (0,\log(1+\sqrt{3})]$ the $j$-metric ball $\Ball{j}{x}{r}$ is close-to-convex and in the case $n = 2$ for $r > \log(1+\sqrt{3})$ it is not close-to-convex.
\end{theorem}
\begin{proof}
  By symmetry of $G$ we may assume $x = e_1$. By \cite[Theorem 1.1]{k1} $\Ball{j}{x}{r}$ is starlike with respect to $x$ and therefore also close-to-convex for $r \le \log (1+\sqrt{2})$. Therefore we assume $r > \log (1+\sqrt{2})$. Now by the definition of the $j$-metric \cite[proof of Theorem 3.1]{k1}
  \[
    \Ball{j}{x}{r} = B^n(e_1,e^r-1) \setminus \overline{B^n} (c,s),
  \]
  where
  \[
    c = \frac{e_1}{e^r(2-e^r)} \quad \textrm{and} \quad s = \frac{e^r-1}{e^r(e^r-2)}.
  \]
  By geometry $\Ball{j}{x}{r}$ is close-to-convex if and only if $r \le r_0$, where $r_0$ is such that
  \begin{equation}\label{pythagorascondition}
    \left( \frac{e^{r_0}-1}{e^{r_0}(e^{r_0}-2)} \right)^2 +\left( 1-\frac{1}{e^{r_0}(2-e^{r_0})} \right)^2 = (e^{r_0}-1)^2.
  \end{equation}
  Equality (\ref{pythagorascondition}) is equivalent to $r_0 = \log (1+\sqrt{3})$ and the assertion follows.
\end{proof}

\begin{remark}
  Let us consider the domain $G = \PS$. Theorem \ref{jforPS} gives a sharp radius of close-to-convexity for $j$-metric disks in the case $n = 2$. In the case $n > 2$ the radius $\log (1+\sqrt{3})$ is presumably not sharp.

  We show that for $r > \log 3$ the $j$-metric balls $\Ball{j}{x}{r}$ are not close-to-convex. By geometry we may assume $x = e_1$ and by the proof of Theorem \ref{jforPS} we know that $\Ball{j}{x}{r} = B^n(x,e^r-1) \setminus \overline{B^n} (c,s)$ for some $c \in G$ and $s > 0$. If $r > \log 3$, then $e^r-1 > 2$, $c \in (-e_1/3,0)$ and $s \in (0,2/3)$. Therefore, $B^n (c,s) \subset B^n(x,e^r-1)$ and $j$-metric balls cannot be close-to-convex for $r > \log 3$.

  We show next that $j$-metric balls $\Ball{j}{x}{r}$ are close-to-convex in $G = \R^3 \setminus \{ 0 \}$ whenever $r \le \log 3$. Fix $x \in G$ and $r \in (\log(1+\sqrt{3}),\log 3]$. By the proof of Theorem \ref{jforPS} we know that $\Ball{j}{x}{r} = B_1 \setminus \overline{B_2}$ for Euclidean balls $B_1$ and $B_2$. We cover $\R^3 \setminus \Ball{j}{x}{r}$ with non-intersecting half-lines and we begin with $\overline{B_2}$. We cover $\overline{B_2}$ with non-intersecting half-lines consisting of cut hyperboloids of one sheet with semi-minor axis $\{ z = t e_1, \, t \in \R \}$. The rest of $\R^3 \setminus \Ball{j}{x}{r}$ can be covered with non-intersecting half-lines orthogonal to $\{ z = t e_1, \, t \in \R \}$. Since $\R^3 \setminus \Ball{j}{x}{r}$ can be covered with non-intersecting half-lines, the $j$-ball $\Ball{j}{x}{r}$ is close-to-convex.
\end{remark}

The following lemma gives an example of close-to-convex domains.
\begin{lemma}\label{balllemma}
  Let $r > 0$ and $x_i \in \Rn$, $r_i > 0$ be such that $|x_i| \ge r/\sqrt{2}$, $r_i < |x_i|$ and $\sqrt{r_i^2+|x_i|^2} \ge r$ whenever $|x_i| < r$. Then
  \[
    B = B^n(0,r) \setminus \bigcup_{i=1}^\infty \overline{B_i}
  \]
  is close-to-convex and connected, where $B_i = B^n(x_i,r_i)$.
\end{lemma}
\begin{proof}
  The assertion is clear by geometry and the selection of points $x_i$ and radii $r_i$.
\end{proof}

Now we are ready to find the radius of close-to-convexity of $j$-metric balls in a general domain.

\begin{theorem}\label{jballsinG}
  Let $G \subsetneq \Rn$ be a domain and $x \in G$. For $r \in (0,\log(1+\sqrt{3})]$ the $j$-metric ball $\Ball{j}{x}{r}$ is close-to-convex and connected.
\end{theorem}
\begin{proof}
  By \cite[Theorem 1.1]{k1} the $j$-metric ball $\Ball{j}{x}{r}$ is starlike with respect to $x$ and thus close-to-convex for $r \le \log(1+\sqrt{2})$. Therefore we assume $r > \log(1+\sqrt{2})$. By \cite[(4.1)]{k1}
  \begin{equation}\label{intersection}
    \Ball{j}{x}{r} = \bigcap_{z \in \partial G} \Ball{j_{\Rn \setminus \{ z \}}}{x}{r}.
  \end{equation}
  We will prove that $\Ball{j}{x}{r}$ is as $B$ in Lemma \ref{balllemma}. By (\ref{intersection}) $\Ball{j}{x}{r} = B^n(x,r_0) \setminus \bigcup_{i=1}^\infty \overline{B^n}(x_i,r_i)$, where
  \begin{eqnarray}
    r_0 & = & d(x) (e^r-1),\label{r0}\\
    |x_i| & = & c \left( 1+\frac{1}{e^r(e^r-2)} \right) \textnormal{ and}\label{xi}\\
    r_i & = & c \frac{e^r-1}{e^r(e^r-2)}\label{ri}
  \end{eqnarray}
  for some $c \ge d(x)$.

  Firstly, we have $|x_i| \ge r_0/\sqrt{2}$ because by (\ref{r0}) and (\ref{xi})
  \[
    \frac{r_0}{|x_i|} \le \frac{e^r-1}{1+\frac{1}{e^r(e^r-2)}} = e^r-1+\frac{1}{1-e^r} \le \sqrt{3}-\frac{1}{\sqrt{3}} \le \sqrt{2},
  \]
  where the second inequality follows from the fact that the function $e^r-1+1/(1-e^r)$ is increasing on $(0,\infty)$ and $r \le \log(1+\sqrt{3})$.

  Secondly, we have $r_i < |x_i|$ since by (\ref{xi}) and (\ref{ri})
  \[
    \frac{r_i}{c} = \frac{e^r-1}{e^r(e^r-2)} = \frac{1}{e^r(e^r-2)}+\frac{1}{e^r} < \frac{1}{e^r(e^r-2)}+1 = \frac{|x_i|}{c}.
  \]

  Finally, we have $\sqrt{r_i^2+|x_i|^2} \ge r_0$ because by (\ref{r0}), (\ref{xi}) and (\ref{ri})
  \[
    r_i^2+|x_i|^2 = \frac{(e^r-1)^2(e^{2r}-2e^r+2)}{e^{2r}(e^r-2)^2} \ge 3 \ge (e^r-1)^2 = r_0^2,
  \]
  where the first inequality follows from the fact that the function $((e^r-1)^2(e^{2r}-2e^r+2))/(e^{2r}(e^r-2)^2)$ is decreasing on $(\log(1+\sqrt{2}),\infty)$. Now the assertion follows from Lemma \ref{balllemma}.
\end{proof}

The close-to-convexity bound of Theorem \ref{jballsinG} is illustrated in two simple domains in Figure \ref{jdiskfig}.

\begin{figure}[htp]
  \begin{center}
    \includegraphics[height=5cm]{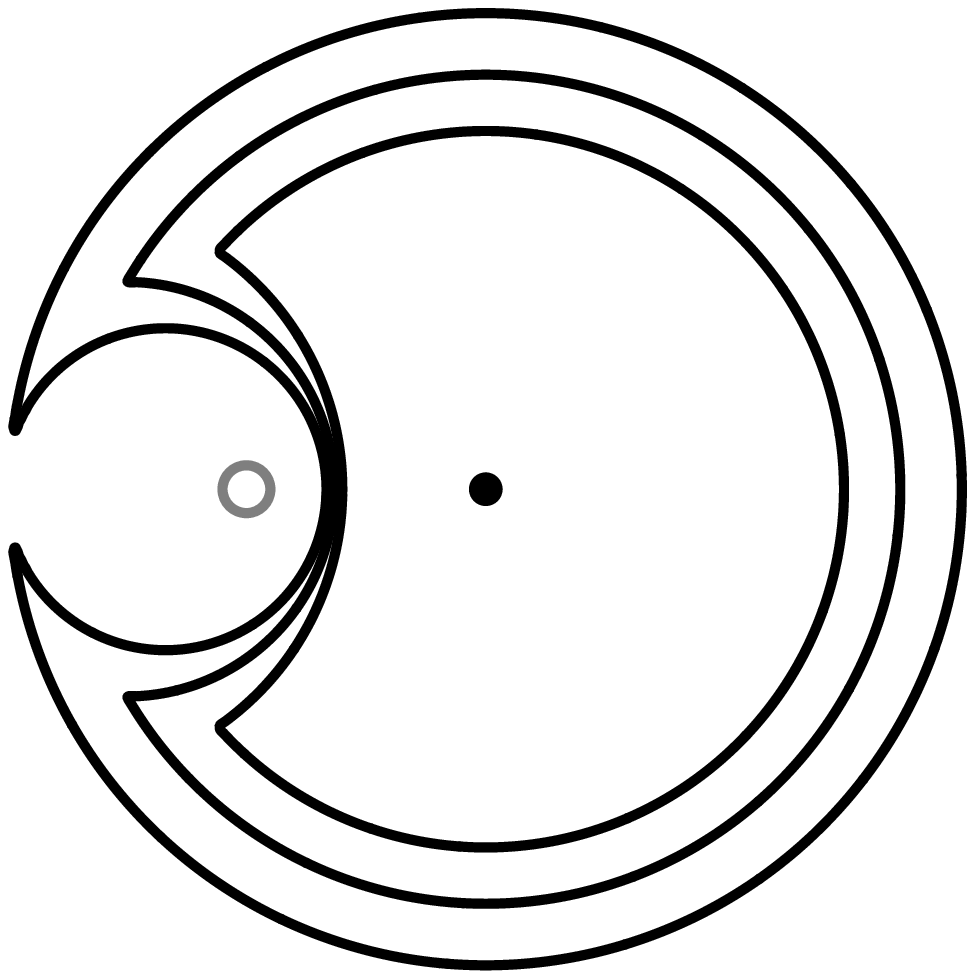}\hspace{1cm}
    \includegraphics[height=5cm]{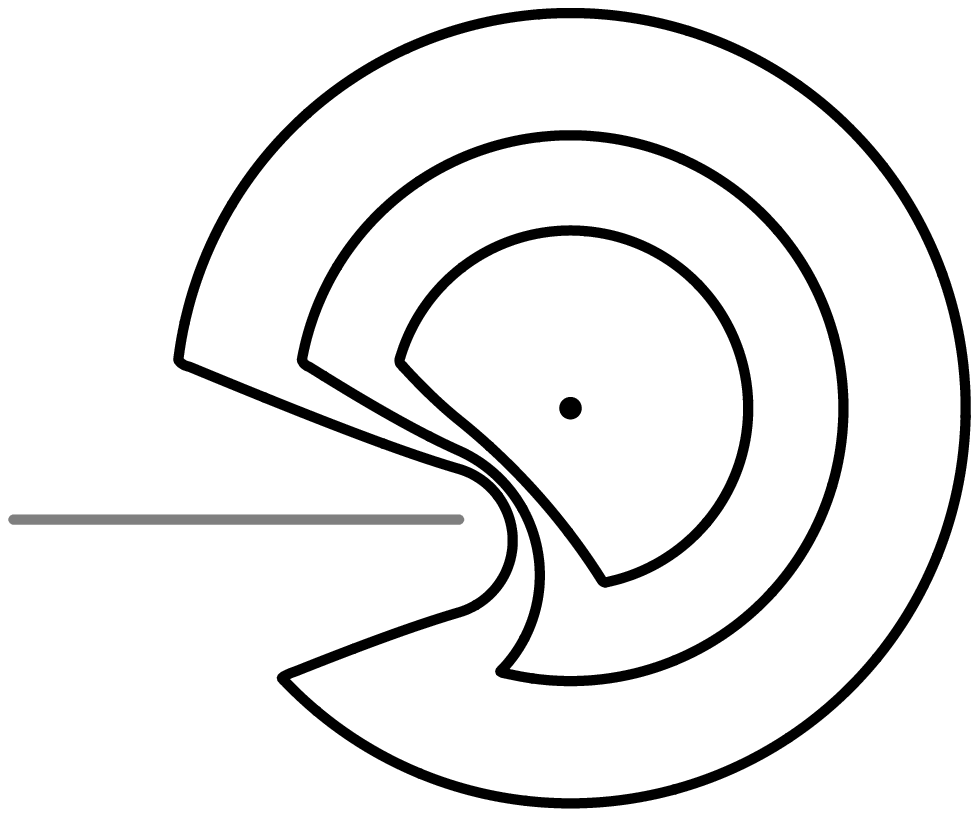}
    \caption{An example of $j$-metric disks in a punctured plane and in a slit plane. The radii of the disks are $\log(1+\sqrt{3})-r_0$, $\log(1+\sqrt{3})$ and $\log(1+\sqrt{3})+r_0$ for small $r_0 > 0$.\label{jdiskfig}}
  \end{center}
\end{figure}

\begin{remark}
  (1) In \cite[Remark 4.9]{k1} is given an example of a domain such that there exists disconnected $j$-metric balls $\Ball{j}{x}{r}$ whenever $r > \log 3$. We will give an example of a domain where the bound $\log 3$  is improved to $\log (1+\sqrt{3})$, which is also the best possible.

  By Theorem \ref{jballsinG} each $\Ball{j}{x}{r}$ in any domain is close-to-convex and connected for all $r \le \log (1+\sqrt{3})$. However, $\Ball{j}{x}{r}$ need not be connected for $r > \log (1+\sqrt{3})$. An example of such a domain and $j$-metric ball is $\C \setminus \{ -1,1 \}$ and $\Ball{j}{x}{r}$ for $x = \sqrt{3}i$ and $r \in ( \log (1+\sqrt{3}),\log (29/10))$. Now $y = -\sqrt{3}i$ is in $\Ball{j}{x}{r}$, because
  \[
    j(x,y) = \log \left( 1+\frac{2\sqrt{3}}{2} \right) = \log(1+\sqrt{3}) < r.
  \]
  We will show that $z \notin \Ball{j}{x}{r}$ for any $z$ with $\Im z = -1$. Let us denote $z = h-i$ for $h \in \R$. By symmetry of the domain it is sufficient to consider $h \ge 0$. If $h > 1+\sqrt{3}$, then
  \begin{eqnarray*}
    j(x,z) & = & \log \left( 1+\frac{|x-z|}{2} \right) = \log \left( 1+\frac{\sqrt{(1+\sqrt{3})^2+h^2}}{2} \right)\\
    & \ge & \log \left( 1+\frac{\sqrt{(1+\sqrt{3})^2+(1+\sqrt{3})^2}}{2} \right) = \log \left( 1+\frac{\sqrt{3}+1}{\sqrt{2}} \right)\\
    & \ge & \log \frac{29}{10},
  \end{eqnarray*}
  because the function $\sqrt{(1+\sqrt{3})^2+h^2}$ is clearly increasing on $(0,\infty)$.

  For $h \in [0,1+\sqrt{3}]$ we have
  \[
    \log(x,z) = \log \left( 1+\frac{\sqrt{(\sqrt{3}+1)^2+h^2}}{\sqrt{1+(h-1)^2}} \right) = \log \left( 1+\sqrt{\frac{4+2\sqrt{3}+h^2}{h^2-2h+2}} \right)
  \]
  and we denote $f(h) = (4+2\sqrt{3}+h^2)/(h^2-2h+2)$. By a straightforward computation $f'(h) = 0$ on $[0,\sqrt{3}]$ if and only if $h = 2\sqrt{2+\sqrt{3}}-1-\sqrt{3}$. Therefore $f(h) \ge \min \{ f(0),f(2\sqrt{2+\sqrt{3}}-1-\sqrt{3},f(\sqrt{3})) \} = f(0)$ on $[0,\sqrt{3}]$. We conclude that
  \[
    \log(x,z) = \log(1+\sqrt{f(h)}) \ge \log(1+\sqrt{f(0)}) = \log \left( 1+\frac{\sqrt{3}+1}{\sqrt{2}} \right) \ge \log \frac{29}{10}
  \]
  and therefore $z \notin \Ball{j}{x}{r}$.
  
  It can be shown that $\Ball{j}{x}{r}$ is disconnected for all $r \in ( \log (1+\sqrt{3}),\log 3)$ and the example is easy to generalize for higher dimensions $n > 2$.

  (2) For any $m \in \N$ there exists a domain and a $j$-metric ball that has exactly $m$ components \cite{k3}.

  (3) Note that there exists a domain $G \subsetneq \R^3$ and $x \in G$ such that $\Ball{j}{x}{r}$ is not close-to-convex for $r > \log(1+\sqrt{2+\sqrt{2}})$. An example of such a domain and point are $G = \R^3 \setminus \{ a+b i,a-b i\}$, where $a = \sqrt{1+1/\sqrt{2}}$ and $b = \sqrt{1-1/\sqrt{2}}$, and $x = 0$.

  Let us first consider domain $D = B^n \setminus \overline{B^n(e_1 3/4,1/2)}$. In the case $n = 2$ clearly $D$ is not close-to-convex, but in the case $n = 3$ we may fill the hole of the complement of the domain $D$ by lines that form cut hyperboloids of one sheet with semi-minor axis $\{ z = t e_1, \, t \in \R \}$. Therefore, we can express $\Rn \setminus \overline{D}$ as a union of half-lines.

  In the case on the domain $G$ a similar construction of cut hyperboloids fail, because the two holes of $\R^3 \setminus \overline{G}$ are too close to each other. If we try to fill the holes as in the case of $D$ some of the cut hyperboloids will intersect and $D$ is not close-to-convex.
\end{remark}

\section{Close-to-convexity of quasihyperbolic balls in punctured space}

We define a constant $\lambda$ to be the solution of the equation
\begin{equation}\label{lambda}
  \cos \sqrt{z^2-1} + \sqrt{z^2-1} \sin \sqrt{z^2-1} = 0
\end{equation}
for $z \in (2,\pi)$. The next proposition shows that (\ref{lambda}) has exactly one solution and its numerical approximation is $\lambda \approx 2.97169$.

\begin{proposition}
  The function
  \[
    f(z) = \cos \sqrt{z^2-1} + \sqrt{z^2-1} \sin \sqrt{z^2-1}
  \]
  has exactly one zero on $(2,\pi)$.
\end{proposition}
\begin{proof}
  By a simple computation $f'(z) = z \cos \sqrt{z^2-1} < 0$, because by assumption $\sqrt{z^2-1} \in (\sqrt{3},\sqrt{\pi^2-1}) \subset (\pi/2,\pi)$. Therefore $f(z)$ is continuous and strictly decreasing on $(2,\pi)$. Because $f(2) > 0$ and $f(\pi) < 0$, the function $f(z)$ has exactly one zero on $(2,\pi)$.
\end{proof}

We will find the radius of close-to-convexity of quasihyperbolic disks in a punctured plane.
\begin{theorem}\label{ctcthmforPP}
  Let $G = \PP$, $r > 0$ and $x \in G$. The quasihyperbolic disk $\Ball{k}{x}{r}$ is close-to-convex for $r \le \lambda$ is not close-to-convex for $r > \lambda$.
\end{theorem}
\begin{proof}
  We may assume that $r \in (2,\pi]$, because for $r \in (0,2]$ the quasihyperbolic disk is starlike with respect to $x$ \cite[Theorem 1.1 (2)]{k2} and thus close-to-convex and for $r > \pi$ the quasihyperbolic disk is not simply connected \cite[Remark 4.8]{k2} and thus not close-to-convex. By symmetry of $G$ we may assume $x = e_1$ and it is sufficient to consider the upper half $U$ of the boundary $\partial \Ball{k}{x}{r}$. By (\ref{moeq})
  \[
    U = \{ y = y(s) \in \R^2 \colon y(s) = (e^s \cos \p(s),e^s \sin \p(s)) \},
  \]
  where $\p(s) = \sqrt{r^2-s^2}$ and $s \in [-r,r]$. Since $\p'(s) = -s/\p(s)$ we obtain for $s \in (-r,r)$
  \[
    y'(s) = \frac{e^s(a(s),b(s))}{\p(s)},
  \]
  where $a(s) = \p(s)\cos \p(s) + s \sin \p(s)$ and $b(s) = \p(s) \sin \p(s) - s \cos \p(s)$. To prove the close-to-convexity we need to find the largest possible $r$ such that the tangent vector $(a(s),b(s))$ of $U$ is parallel to the real axis exactly once for $s \in (-r,0)$. Since
  \[
    b'(s) = \frac{-(1+s)a(s)}{\varphi(s)}
  \]
  and $a(s) < 0$ for all $s \in (-r,0)$ we have $b'(s) = 0$ if and only if $s = -1$. Therefore we want to find the greatest $r$ such that $b(-1) = 0$, but this condition is equivalent to (\ref{lambda}) and the assertion follows.
\end{proof}

The close-to-convexity bound of Theorem \ref{ctcthmforPP} is illustrated in Figure \ref{kdiskfig}.

\begin{figure}[htp]
  \begin{center}
    \includegraphics[height=5cm]{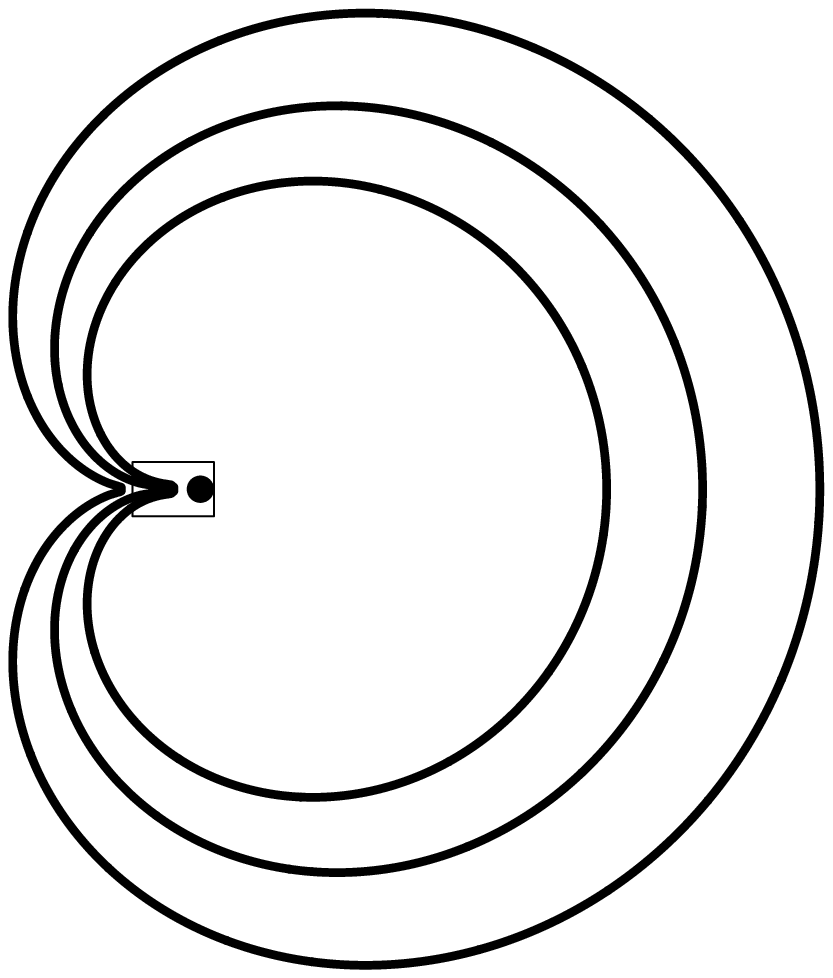}\hspace{1cm}
    \includegraphics[height=5cm]{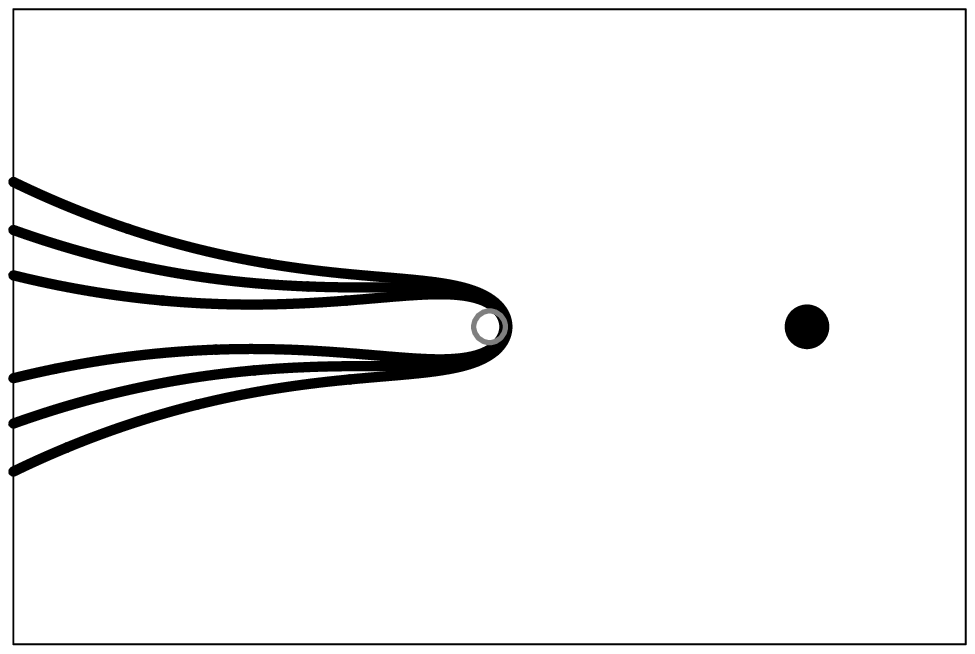}
    \caption{An example of quasihyperbolic disks in punctured plane. The radii of the disks are $\lambda-r_0$, $\lambda$ and $\lambda+r_0$ for small $r_0 > 0$.\label{kdiskfig}}
  \end{center}
\end{figure}

\begin{lemma}\label{symmetrylemma}
  If the domain $G \subset \Rn$ is invariant under rotation about a line $l$ and $G \cap L$ is close-to-convex for every plane $L$ with $l \subset L$, then $G$ is close-to-convex.
\end{lemma}
\begin{proof}
  By assumption each plane $L$ is symmetric about the line $l$ and therefore $(\Rn \setminus G) \cap L$ can be covered by non-intersecting half-lines so that if a half-line $h$ of the cover intersects $l$ then $h \subset l$. Now we can combine the covers of $(\Rn \setminus G) \cap L$ for all $L$ to construct a cover consisting of non-intersecting half-lines for $\Rn \setminus G$. Therefore $G$ is close-to-convex.
\end{proof}

\begin{corollary}\label{kballsinPS}
  Let $G = \PS$, $r > 0$ and $x \in G$. The quasihyperbolic ball $\Ball{k}{x}{r}$ is close-to-convex for $r \le \lambda$.
\end{corollary}
\begin{proof}
  The assertion follows from Theorem \ref{ctcthmforPP} and Lemma \ref{symmetrylemma}.
\end{proof}

\begin{remark}
  It can be shown by numerical computation that the quasihyperbolic balls cannot be close-to-convex in $\PS$ for radius greater than or equal to $3.1116$. In this case there exists points $y \in \partial \Ball{k}{x}{3.1116}$ such that $y$ is not an end point of any half-lines in $\Rn \setminus \overline{\Ball{k}{x}{3.1116}}$.

  Quasihyperbolic balls $\Ball{k}{x}{r}$ in $\PS$ are not close-to-convex for $r > \pi$, because $\Rn \setminus \overline{\Ball{k}{x}{r}}$ consists of two disconnected component and one of the components is bounded.
\end{remark}

\begin{proof}[Proof of Theorem \ref{mainthm}]
  The assertion follows from Theorem \ref{jforPS}, Theorem \ref{jballsinG}, Theorem \ref{ctcthmforPP} and Corollary \ref{kballsinPS}.
\end{proof}

In Table \ref{ksummary} and Table \ref{jsummary} the known radii of convexity, starlikeness and close-to-convexity for the quasihyperbolic and the $j$-metric balls are presented.
\begin{table}[h!]
  \begin{center}\begin{tabular}{c|c|c|c}
    domain & convex & starlike w.r.t $x$ & close-to-convex\\
    \hline
    $\PP$ & 1 \cite{k1} & $\kappa \approx 2.83$ \cite{k1} & $\lambda \approx 2.97$ (Thm \ref{ctcthmforPP})\\
    $\PS$ & 1 \cite{k1} & $\kappa \approx 2.83$ \cite{k1} & $\lambda^* \approx 2.97$ (Cor. \ref{kballsinPS})\\
    convex & $\infty$ \cite{mv} & $\infty$ \cite{mv} & $\infty$ \cite{mv}\\
    starlike w.r.t. $x$ & ? & $\infty$ \cite{k1} & $\infty$ \cite{k1}\\
    general ($n = 2$) & 1 \cite{v2} & $\pi/2^*$ \cite{v1} & $\pi/2^*$ \cite{v1}\\
    general ($n \ge 2$) & ? & $\pi/2^*$ \cite{v1} & $\pi/2^*$ \cite{v1}\\
  \end{tabular}\end{center}
  \caption{\label{ksummary} The known radii of convexity, starlikeness and close-to-convexity for the quasihyperbolic balls $\Ball{k}{x}{r}$. Notation $r^*$ means that the radius $r$ is not sharp.}
\end{table}
\begin{table}[h!]
  \begin{center}\begin{tabular}{c|c|c|c}
    domain & convex & starlike w.r.t $x$ & close-to-convex\\
    \hline
    convex & $\infty$ \cite{k2} & $\infty$ \cite{k2} & $\infty$ \cite{k2}\\
    starlike w.r.t. $x$ & $\log 2$ \cite{k2} & $\infty$ \cite{k2} & $\infty$ \cite{k2}\\
    general ($n = 2$) & $\log 2$ \cite{k2} & $\log (1+\sqrt{2})$ \cite{k2} & $\log (1+\sqrt{3})$ (Thm \ref{jforPS})\\
    general ($n \ge 2$) & $\log 2$ \cite{k2} & $\log (1+\sqrt{2})$ \cite{k2} & $\log (1+\sqrt{3})^*$ (Thm \ref{jballsinG})\\
  \end{tabular}\end{center}
  \caption{\label{jsummary} The known radii of convexity, starlikeness and close-to-convexity for the $j$-metric balls $\Ball{j}{x}{r}$. Notation $r^*$ means that the radius $r$ is not sharp.}
\end{table}

Finally we pose some open problems concerning the close-to-convexity of the quasihyperbolic and the $j$-metric balls. Naturally similar questions can be asked for other hyperbolic type metrics.
\begin{itemize}
\item Let $G \subsetneq \Rn$ be a domain and $n > 2$. What is the sharp bound of radius $r_j$ such that $\Ball{j}{x}{r}$ is close-to-convex for all $x \in G$ and $r \in (0,r_j]$?

  \item Are the $j$-metric balls close-to-convex in close-to-convex domains for all radii $r > 0$?

  \item Let $G = \PS$ and $n > 2$. What is the sharp bound of radius $r_k$ such that $\Ball{k}{x}{r}$ is close-to-convex for all $x \in G$ and $r \in (0,r_k]$?

  \item Are the quasihyperbolic balls $\Ball{k}{x}{r}$ close-to-convex for $r \in (0,\lambda]$ in any domain $G \subsetneq \Rn$?

  \item Are the quasihyperbolic balls close-to-convex in close-to-convex domains for all radii $r > 0$?
\end{itemize}

\textbf{Acknowledgements.} The work was supported by the Finnish National Graduate School in Mathematical Analysis and Its Applications.


{\noindent Department of Mathematics\\
University of Turku\\
FI-20014\\
FINLAND\\
e-mail: riku.klen@utu.fi}

\end{document}